\documentclass{aptpub}

\authornames{F. Buono and M. Longobardi} 
\shorttitle{Varentropy of past lifetimes} 

\numberwithin{equation}{section}  

\begin{document}

\title{Varentropy of past lifetimes} 

\authorone[Università degli Studi di Napoli Federico II]{Francesco Buono} 
\addressone{Dipartimento di Matematica e Applicazioni "Renato Caccioppoli", Università degli Studi di Napoli Federico II, Via Cintia, 80126, Naples, Italy} 
\authortwo[Università degli Studi di Napoli Federico II]{Maria Longobardi} 
\addresstwo{Dipartimento di Biologia, Università degli Studi di Napoli Federico II, Via Cintia, 80126, Naples, Italy} 

\begin{abstract}
In Reliability Theory, uncertainty is measured by the Shannon entropy. Recently, in order to analyze the variability of such measure, varentropy has been introduced and studied. In this paper we define a new concept of varentropy for past lifetime distributions as the variability of uncertainty in past lifetime. We analyze the relation of this varentropy with the corresponding past entropy and other concepts of interest in reliability and we study the effect of linear and monotonic transformations on it. Finally, we evaluate past varentropy for random variables following a proportional reversed hazard rate model.
\end{abstract}
\bigskip


\section{Introduction}

Let $X$ be an absolutely continuous random variable with support $D\subseteq(0,+\infty)$, describing the lifetime of a living being or of an item. Let $F$ be its cumulative distribution function (cdf), $\overline F$ its survival function (sf) and $f$ its probability density function (pdf).

In reliability, in information theory, in finance, in theoretical neurobiology and in several other fields, a well-known measure of uncertainty for $X$ is the differential entropy or Shannon information measure, defined as
\begin{equation*}
H(X)=\mathbb E[-\log f(X)]=-\int_0^{+\infty} f(x)\log f(x) \mathrm dx,
\end{equation*}
where $\log$ is the natural logarithm  (see Shannon \cite{shannon}\textbf{il}).

Ebrahimi \cite{ebrahimi} and Muliere et al. \cite{muliere} studied the residual entropy as measure of information for the residual lifetime $X_t=(X-t|X\ge t)$. The cdf of the residual lifetime is expressed as
\begin{equation*}
f_{X_t}(x)=\frac{f(x+t)}{\overline F(t)}, \ \ \ \ \ \ x>0,
\end{equation*}
whereas the residual entropy can be written as
\begin{equation*}
H(X_t)=-\int_t^{+\infty} \frac{f(x)}{\overline F(t)}\log \frac{f(x)}{\overline F(t)}\mathrm dx.
\end{equation*}

For all $t\in D$, the past lifetime of $X$ is $_tX=(X|X\leq t)$, with cdf and pdf expressed as
\begin{equation*}
F_{_tX}(x)=\frac{F(x)}{F(t)}, \ \ \ \ \ \ f_{_tX}(x)=\frac{f(x)}{F(t)} \ \ \ \ \ \ x\in(0,t)
\end{equation*}
and the inactivity time of $X$ is $X_{(t)}=(t-X|X\leq t)$, where its sf and pdf are 
\begin{equation*}
\overline F_{(t)}(x)=\frac{F(t-x)}{F(t)}, \ \ \ \ \ \ f_{(t)}(x)=\frac{f(t-x)}{F(t)} \ \ \ \ \ \ x\in(0,t)
\end{equation*}
(see Li and Zuo \cite{li} for details).

Di Crescenzo and Longobardi \cite{dicrescenzo} studied the past entropy to measure the uncertainty about the past lifetime. The past entropy is defined as
\begin{equation}
\label{eq5}
H(_tX)=\mathbb E[-\log f_{_tX}(_tX)]=-\int_0^t \frac{f(x)}{F(t)} \log \frac{f(x)}{F(t)} \mathrm dx.
\end{equation}

The mean inactivity time is a function of a great importance in the context of reliability theory. It is the mean value of $X_{(t)}$ and it is expressed as
\begin{equation}
\label{eq16}
m^*(t)=\mathbb E(X_{(t)})=\frac{1}{F(t)}\int_0^t F(t-x) \mathrm dx=\frac{1}{F(t)}\int_0^t F(x)\mathrm dx,
\end{equation}
(see Izadkhan \cite{izadkhan} for further details).

Notice that the past entropy is also the entropy of the inactivity time, that is, $H(_tX)=H(X_{(t)})$.

\noindent We define, for $t\in D$, the reversed hazard rate function of $X$ as 
\begin{equation}
\label{eq7}
q(t)=\lim_{\Delta t\to 0^+} \frac{1}{\Delta t}\mathbb P(X\ge t-\Delta t| X\leq t)=\frac{f(t)}{F(t)}.
\end{equation}
This function is the instantaneous failure rate occurring immediately before the time point $t$, i.e., the failure occurs just before the time point $t$, given that the unit has not survived longer than $t$ (see Block and Savits \cite{block} and Finkelstein \cite{finkelstein} for more details about the reversed hazard rate function).  Recall also that the cumulative reversed hazard rate function s defined as
\begin{equation}
\label{eq4}
\Lambda^*(t)=\int_t^{+\infty} q(x) \mathrm dx = -\log F(t).
\end{equation}
 
The past entropy can also be expressed in different ways (see \cite{dicrescenzo}):
\begin{eqnarray}
\label{eq1}
H(_tX)&=& -\Lambda^*(t)-\frac{1}{F(t)}\int_0^t f(x)\log f(x)\mathrm dx \\
&=& 1-\frac{1}{F(t)}\int_0^t f(x)\log q(x)\mathrm dx.
\end{eqnarray}
By differentiating both sides of \eqref{eq1}, we get the following expression for the derivative of the past entropy:
\begin{equation}
\label{eq8}
H'(_tX)=q(t)[1-H(_tX)-\log q(t)].
\end{equation}
In recent literature several new definitions of entropy of a continuos distribution have been proposed: weighted version and its residual and past ones (see \cite{DiLo2006}) and many different version of cumulative entropies and their applications (see for instance \cite{CaLoAh}, \cite{CaLoNa}, \cite{CaLoPsa}, \cite{DiLo2009a},  \cite{DiLo2009b}, \cite{DiLo2012}, \cite{DiLo2013} and \cite{Lo2014}).

We recall, also, that the distribution function, the reversed hazard rate function and the mean inactivity lifetimes are linked by the following relations:
\begin{align}
\label{eq11}
& F(t)=\exp\left[-\int_t^{+\infty} q(x) \mathrm dx \right] \\
& q(t)=\frac{1-(m^*(t))'}{m^*(t)} .
\end{align}

In Section 2, the definitions of varentropy and residual varentropy are recalled and a new varentropy of past lifetime distributions is defined. Moreover, some examples of past varentropy are exhibited and some relations with other important concepts of reliability theory are given. The behaviour of such a measure under certain transformations is also described.

In Section 3, we give some bounds for past varentropy. Finally, we consider the proportional reversed hazard rate model and evaluate the past varentropy of the family of random variables following this model.

\section{Past varentropy}

The differential entropy can be seen as expectation of the information contained in $X$. 
In order to measure the concentration of this information content around the entropy, recently a new definition called varentropy was given (see also Bobkov and Madiman \cite{bobkov})). The varentropy is defined as
\begin{eqnarray}
\label{eq3}
V(X)&=&\mathrm{Var}(-\log f(X))= \mathrm{Var}(\log f(X))=\mathbb E\left[(\log f(X))^2\right]-[H(X)]^2 \\
&=&\int_0^{+\infty} f(x)[\log f(x)]^2 \mathrm dx-\left[\int_0^{+\infty} f(x)\log f(x)\mathrm dx\right]^2.
\end{eqnarray}
In the literature there are several papers about the study of the varentropy and its properties, we may refer to Arikan \cite{arikan} and Madiman and Wang \cite{madiman}.
In the discrete case, we can write the entropy and the varentropy of a random variable $X$ with support  $\{x_i, i \in I\}$ and discrete density function $p_X$, as 
$$H(X)=-\sum_{i \in I} p_X(x_i) \log p_X(x_i)$$
and
$$V(X)=\sum_{i \in I} p_X(x_i) [\log p_X(x_i)]^2- [H(X)]^2,$$
respectively.

Di Crescenzo and Paolillo \cite{dicrescpao} extended the definition of the varentropy \eqref{eq3} to the residual lifetimes introducing the residual varentropy that is
\begin{eqnarray*}
V(X_t)&=&\mathrm{Var}(-\log f_{X_t}(X_t))= \mathrm{Var}(\log f_{X_t}(X_t))=\mathbb E\left[(\log f_{X_t}(X_t))^2\right]-[H(X_t)]^2 \\
&=&\int_t^{+\infty} \frac{f(x)}{\overline F(t)}\left[\log \frac{f(x)}{\overline F(t)}\right]^2 \mathrm dx-\left[\int_t^{+\infty} \frac{f(x)}{\overline F(t)}\log \frac{f(x)}{\overline F(t)}\mathrm dx\right]^2 \\
&=& \frac{1}{\overline F(t)}\int_t^{+\infty} f(x)(\log f(x))^2 \mathrm dx - (\Lambda (t)+H(X_t))^2,
\end{eqnarray*}
where, in analogy with \eqref{eq4}, $\Lambda(t)=-\log\overline F(t)$.

Recently, a great importance is given to the study of past lifetime distributions, because, in many realistic situation, it is reasonable that uncertainty is related to the past. Let us suppose to have a system observed at a certain time $t$ and then it is found failed. So the uncertainty is referred to the past and the concentration around it can be investigated.
For this reason, we introduce and study the past varentropy, i.e., the varentropy of the past lifetime. We define, for all $t\in D$, the past varentropy as
\begin{eqnarray}
\nonumber
V(_tX)&=&\mathrm{Var}(-\log f_{_tX}(_tX))= \mathrm{Var}(\log f_{_tX}(_tX))=\mathbb E\left[(\log f_{_tX}(_tX))^2\right]-[H(_tX)]^2 \\
\label{eq13}
&=&\int_0^t \frac{f(x)}{F(t)}\left[\log \frac{f(x)}{F(t)}\right]^2 \mathrm dx-\left[\int_0^t \frac{f(x)}{F(t)}\log \frac{f(x)}{F(t)}\mathrm dx\right]^2 \\
\label{eq6}
&=& \frac{1}{F(t)}\int_0^t f(x)(\log f(x))^2 \mathrm dx - (\Lambda^*(t)+H(_tX))^2,
\end{eqnarray}
where $\Lambda^*(t)$ and $H(_tX)$ are given in \eqref{eq4} and \eqref{eq5}, respectively.
We note that the past varentropy is also the varentropy of the inactivity time, that is, $V(_tX)=V(X_{(t)})$.

In the following example we evaluate the past entropy and the past varentropy for some distributions of interest in reliability theory.

\begin{ex}
\begin{itemize}
\item Let $X$ be a random variable with uniform distribution over $(0,b)$, i.e., $X\sim U(0,b)$, $b>0$. Hence, for $t\in (0,b)$ we have
\begin{align*}
&H(_tX)=\log t \\
&V(_tX)=0.
\end{align*}

\item Let $X$ be a random variable with exponential distribution, i.e., $X\sim Exp(\lambda)$, $\lambda>0$. Then, for $t>0$ we have
\begin{align*}
H(_tX)&=1+\log \left(\frac{1-\mathrm e^{-\lambda t}}{\lambda}\right)-\frac{\lambda t \mathrm e^{-\lambda t}}{1-\mathrm e^{-\lambda t}} \\
V(_tX)&=1-2\lambda^2\log \lambda+2\log \lambda\\
&+\frac{\mathrm e^{-\lambda t}}{1-\mathrm e^{-\lambda t}}\left[-\lambda^2 t^2+2\lambda^3t\log \lambda-2\lambda t\log \lambda-\frac{\lambda^2t^2\mathrm e^{-\lambda t}}{1-\mathrm e^{-\lambda t}}\right].
\end{align*}

\item Let $X$ be a random variable such that $f(x)=2x$ and $F(x)=x^2$, $x\in(0,1)$. Hence, for $t\in (0,1)$ we have
\begin{align*}
&H(_tX)=\frac{1}{2}+\log \frac{t}{2} \\
&V(_tX)=\frac{1}{4}+4\log2\log t+\log^2 t^2-2\log 2t\log t^2.
\end{align*}
\end{itemize}
\end{ex}

In the next proposition, we obtain an expression for the derivative of the past varentropy.

\begin{prop}
\label{prop1}
For all $t\in D$, the derivative of the past varentropy is
\begin{equation*}
V'(_tX)=-q(t)\left[V(_tX)-(H(_tX)+\log q(t))^2\right].
\end{equation*}
\end{prop}

\begin{proof}
By differentiating both sides of \eqref{eq6}, we get
\begin{eqnarray}
\nonumber
V'(_tX)&=&=\frac{q(t)}{F(t)}\int_0^t f(x)(\log f(x))^2\mathrm dx+q(t)(\log f(t))^2\\
\label{eq9}
&=&-2(\Lambda^*(t)+H(_tX))(-q(t)+H'(_tX)),
\end{eqnarray}
where $q(t)$ is defined in \eqref{eq7}. Hence, recalling \eqref{eq8} and \eqref{eq6}, from \eqref{eq9} we get
\begin{eqnarray*}
V'(_tX)&=&-q(t)\left[V(_tX)+(\Lambda^*(t)+H(_tX))^2-(\log f(t))^2 \right. \\
& &\left. -2(\Lambda^*(t)+H(_tX))(H(_tX)+\log q(t))\right],
\end{eqnarray*}
and, after some calculations, we get the thesis.
\end{proof}

It is useful to study the monotonicity of this function and to establish if, under some conditions, the past varentropy can be constant.
 
\begin{thm}
(i) Let the past varentropy $V(_tX)$ be constant, i.e., $V(_tX)=v\ge0$, for all $t\in D$. Then
\begin{equation*}
|H(_tX)+\log q(t)|=\sqrt v, \ \ \ \ \ \forall t\in D.
\end{equation*}

\noindent (ii) Let $c\in\mathbb R$. If 
\begin{equation}
\label{eq2}
H(_tX)+\log q(t)=c \ \ \ \ \ \forall t\in D
\end{equation}
then 
\begin{equation*}
V(_tX)=c^2+\frac{V(X)-c^2}{F(t)}\ \ \ \ \ \forall t\in D.
\end{equation*}
\end{thm}

\begin{proof}
$(i)$ From Proposition \ref{prop1}, we get that if the past varentropy is constant, then
\begin{equation*}
(H(_tX)+\log q(t))^2=v\ge0,
\end{equation*}
and so
\begin{equation*}
|H(_tX)+\log q(t)|=\sqrt v, \ \ \ \ \ \forall t\in D.
\end{equation*}

\noindent $(ii)$ From Proposition \ref{prop1} and from the assumption \eqref{eq2}, we get
\begin{equation*}
V'(_tX)=-q(t)(V(_tX)-c^2)
\end{equation*}
with the boundary condition
\begin{equation*}
\lim_{t\to \sup D} V(_tX)=V(X).
\end{equation*}
Solving the above linear ordinary differential equation we get the thesis.
\end{proof}

\begin{rem}
We refer to Corollary 2.2 of Kunda et al. \cite{kunda} for distributions which satisfy the relation expressed in \eqref{eq2}.
\end{rem}

Recall that the generalized reversed hazard rate is defined, for $\alpha\in\mathbb R,$ as
\begin{equation}
\label{eq10}
q_{\alpha}(t)=\frac{f(t)}{[F(t)]^{1-\alpha}}, \ \ \ \ \ t\in D
\end{equation}
(see Buono et al. \cite{buono}).
Notice that, by choosing $\alpha=0$ in \eqref{eq10}, we get $q_0(t)=q(t)$, i.e., $q_0$ is the usual reversed hazard rate function. In the following theorem, we study a necessary and sufficient condition in terms of past entropy, under which the generalized reversed hazard rate function is constant.

\begin{thm}
Let $X$ be an absolutely continuous random variable with support $D$ and let $c\in\mathbb R$. The generalized reversed hazard rate function of $X$ with parameter $1-c$, that is, $q_{1-c}(t)$, is constant, such that 
$$q_{1-c}(t)=\frac{f(t)}{[F(t)]^c}=\mathrm e^{c-H(X)}, \;\forall t\in D$$
 if, and olnly if,  \eqref{eq2} holds.
\end{thm}

\begin{proof}
Let us suppose that $q_{1-c}(t)=\mathrm e^{c-H(X)}, \forall t\in D$. From \eqref{eq7} and \eqref{eq1} we get
\begin{eqnarray*}
H(_tX)+\log q(t)&=&\log f(t)-\frac{1}{F(t)}\int_0^t f(x)\log f(x)\mathrm dx \\
&=&\log f(t)+\frac{1}{F(t)}\left[H(X)+\int_t^{+\infty} f(x)\log f(x)\mathrm dx\right].
\end{eqnarray*}
From the hypothesis, we get
\begin{equation*}
\int_t^{+\infty} f(x)\log f(x)\mathrm dx=-H(X)\overline F(t)-cF(t)\log F(t),
\end{equation*}
and so
\begin{equation*}
H(_tX)+\log q(t)=H(X)+\log\frac{f(t)}{[F(t)]^c}=c,
\end{equation*}
i.e., the thesis.

\noindent Conversely, from \eqref{eq8} we obtain
\begin{equation*}
c=H(_tX)+\log q(t)=1-\frac{H'(_tX)}{q(t)},
\end{equation*}
and so
\begin{equation*}
H'(_tX)=(1-c)q(t).
\end{equation*}
By integrating both sides from $t$ and $+\infty$ and recalling \eqref{eq11}, we get
\begin{equation*}
H(X)-H(_tX)=(1-c)\Lambda^*(t).
\end{equation*}
Hence,
\begin{eqnarray*}
c=H(_tX)+\log q(t)&=&H(X)+\log F(t)-\log (F(t))^c+\log q(t) \\
&=& H(X)+\log q_{1-c}(t),
\end{eqnarray*}
and we get the thesis.
\end{proof}

In the following proposition, we study the past varentropy under linear transformations. We recall that if 
\begin{equation*}
Y=aX+b, \ \ \ \ \ a>0, \ b\ge0,
\end{equation*}
then the past entropies of $X$ and $Y$ are connected by (see \cite{dicrescenzo})
\begin{equation}
\label{eq12}
H(_tY)=H\left(_{\frac{t-b}{a}}X\right)+\log a \ \ \ \ \ \forall t.
\end{equation}

\begin{prop}
Let $Y=aX+b$, with $a>0$ and $b\ge0$. Then, for their past varentropies, we have
\begin{equation}
V(_tY)=V\left(_{\frac{t-b}{a}}X\right), \ \ \ \ \forall t.
\end{equation}
\end{prop}

\begin{proof}
From $Y=aX+b$ we know that $F_Y(x)=F_X\left(\frac{x-b}{a}\right)$ and $f_Y(x)=\frac{1}{a}f_X\left(\frac{x-b}{a}\right)$. Hence, from \eqref{eq13} and \eqref{eq12}, we get
\begin{equation}
\label{eq14}
V(_tY)=\int_0^{\frac{t-b}{a}} \frac{f_X(x)}{F_X\left(\frac{t-b}{a}\right)}\left(\log \frac{\frac{1}{a}f_X\left(x\right)}{F_X\left(\frac{t-b}{a}\right)}\right)^2\mathrm dx-\left(H\left(_{\frac{t-b}{a}}X\right)+\log a\right)^2.
\end{equation}
By writing 
$$\log \frac{\frac{1}{a}f_X\left(x\right)}{F_X\left(\frac{t-b}{a}\right)}=\log \frac{f_X\left(x\right)}{F_X\left(\frac{t-b}{a}\right)}-\log a,$$
and developing the two squares in \eqref{eq14}, after some calculations we get the thesis.
\end{proof}

If $\phi$ is any continuous, differentiable and strictly monotonic function and 
\begin{equation*}
Y=\phi(X),
\end{equation*}
then the past entropies of $X$ and $Y$ are related, for all $t>0$, by
\begin{equation}
\label{eq20}
H(_tY)=\begin{cases}H\left(_{\phi^{-1}(t)}X\right)+\mathbb E[\log \phi'(X)|X<\phi^{-1}(t)]   , & \mbox{if } \phi \mbox{ is strictly increasing,}\\
H\left(X_{\phi^{-1}(t)}\right)+\mathbb E[\log (-\phi'(X))|X>\phi^{-1}(t)]  , & \mbox{if } \phi \mbox{ is strictly decreasing}
\end{cases}
\end{equation}
(see \cite{dicrescenzo}).
In the following result,  the past varentropy under monotonic transformations is studied.
\begin{prop}
Let $Y=\phi(X)$, where $\phi$ is a continuous, differentiable and strictly monotonic function. 
\begin{itemize} 
\item [(i)] If $\phi$ is strictly increasing, then
\begin{eqnarray}
\nonumber
V(_tY)&=&V\left(_{\phi^{-1}(t)}X\right)-2\mathbb E\left[\left. \log\frac{f_X(X)}{F_X(\phi^{-1}(t))}\log \phi'(X)\right|X<\phi^{-1}(t)\right] \\ 
\nonumber
& &+\mathrm{Var}[\log \phi'(X)|X<\phi^{-1}(t)]-2H\left(_{\phi^{-1}(t)}X\right)\mathbb[\log \phi'(X)|X<\phi^{-1}(t)].
\end{eqnarray}
\item[(ii)] If $\phi$ is strictly decreasing, then
\begin{eqnarray}
\nonumber
V(_tY)&=&V\left(X_{\phi^{-1}(t)}\right)-2\mathbb E\left[\left. \log\frac{f_X(X)}{\overline F_X(\phi^{-1}(t))}\log (-\phi'(X))\right|X>\phi^{-1}(t)\right] \\ \nonumber & &+\mathrm{Var}[\log (-\phi'(X))|X>\phi^{-1}(t)]-2H\left(X_{\phi^{-1}(t)}\right)\mathbb[\log (-\phi'(X))|X>\phi^{-1}(t)].
\end{eqnarray}
\end{itemize}
\end{prop}

\begin{proof}
Let us suppose $\phi$ strictly increasing. From $Y=\phi(X)$ we know that $F_Y(x)=F_X\left(\phi^{-1}(x)\right)$ and $f_Y(x)=\frac{f_X(\phi^{-1}(x))}{\phi'(\phi^{-1}(x))}$. Hence, from \eqref{eq13} and \eqref{eq20}, we get
\begin{eqnarray*}
V(_tY)&=&\int_0^{\phi^{-1}(t)} \frac{f_X(x)}{F_X\left(\phi^{-1}(t)\right)}\left(\log \frac{f_X\left(x\right)}{F_X\left(\phi^{-1}(t)\right)}-\log\phi'(x)\right)^2\mathrm dx \\
& &-\left[H\left(_{\phi^{-1}(t)}X\right)+\mathbb E[\log \phi'(X)|X<\phi^{-1}(t)]\right]^2.
\end{eqnarray*}
Hence, by developing the two squares in the previous equality, and observing that
\begin{eqnarray*}
\int_0^{\phi^{-1}(t)} \frac{f_X(x)}{F_X\left(\phi^{-1}(t)\right)}\left(\log \phi'(x)\right)^2\mathrm dx&-&\mathbb E^2[\log \phi'(X)|X<\phi^{-1}(t)]\\
&=&\mathrm{Var}[\log \phi'(X)|X<\phi^{-1}(t)],
\end{eqnarray*}
we obtain the result.

The proof of (ii) is similar.
\end{proof}

\section{Bounds for past varentropy}

We recall the definition of variance inactivity time function, that is the variance of the inactivity time $X_{(t)}$ and can be expressed as
\begin{eqnarray*}
(\sigma^2(t))^*=\mathrm {Var}(X_{(t)})=\mathrm{Var}(t-X|X\leq t)=\frac{2}{F(t)}\int_0^t \mathrm dy \int_0^y F(x) \mathrm dx -(m^*(t))^2,
\end{eqnarray*}
where $m^*(t)$ is defined in \eqref{eq16}. (For more properties of the variance inactivity time function, see Kandil et al. \cite{kandil}).

\begin{thm}
Let $X_{(t)}$ be the inactivity time at time $t$ of $X$, and let the mean inactivity time $m^*(t)$ and the variance inactivity time $(\sigma^2(t))^*$ be finite. Then, for all $t\in D$, 
\begin{equation*}
V(_tX)\ge (\sigma^2(t))^*\left[\mathbb E(\omega_{(t)}'(X_{(t)}))\right]^2,
\end{equation*}
where the function $\omega_{(t)}(x)$ is defined by
\begin{equation}
\label{eq18}
(\sigma^2(t))^* \omega_{(t)}(x) f_{X_{(t)}}(x)=\int_0^x (m^*(z)-z)f_{X_{(t)}}(z) \mathrm dz, \ \ \ x>0.
\end{equation}
\end{thm}

\begin{proof}
First of all, notice that $V(_tX)=V(X_{(t)}).$  If $X$ is a random variable with pdf $f$, mean $\mu$ and variance $\sigma^2$, we have 
\begin{equation}
\label{eq17}
\mathrm{Var}[g(X)]\ge \sigma^2 \left[\mathbb E(\omega(X)g'(X))\right]^2,
\end{equation}
where $\omega(x)$ is defined by $\sigma^2 \omega(x)f(x)=\int_0^x (\mu-z)f(z)\mathrm dz$  (see Cacoullos and Papathanasiou \cite{cacoullos}). Hence, by choosing $X_{(t)}$ as the random variable in \eqref{eq17}, and  $g(x)=-\log f_{X_{(t)}}(x)$, we obtain
\begin{equation}
\label{eq19}
\mathrm{Var}(-\log f_{X_{(t)}}(X_{(t)}))=V(_tX)\ge (\sigma^2(t))^*\left[\mathbb E\left(\omega_{(t)}(X_{(t)})\frac{f_{X_{(t)}}'(X_{(t)})}{f_{X_{(t)}}(X_{(t)})}\right)\right]^2.
\end{equation}
By differentiating now both sides of \eqref{eq18}, we get
\begin{equation*}
\omega_{(t)}(x)\frac{f_{X_{(t)}}'(x)}{f_{X_{(t)}}(x)}=\frac{m^*(t)-x}{(\sigma^2(t))^*}-\omega_{(t)}'(x),
\end{equation*}
and then by \eqref{eq19}
\begin{eqnarray*}
V(_tX)&\ge& (\sigma^2(t))^*\left[\mathbb E\left(\frac{m^*(t)-X_{(t)}}{(\sigma^2(t))^*}-\omega_{(t)}'(X_{(t)})\right)\right]^2 \\
&=&(\sigma^2(t))^*\left[\mathbb E(\omega_{(t)}'(X_{(t)}))\right]^2.
\end{eqnarray*}
\end{proof}

\begin{thm}
Let $X$ be a non-negative random variable with support $D$ and log-concave pdf $f$. Then
\begin{equation*}
V(_tX)\leq 1 \ \ \ \ \ \mbox{ for all }\  t\in D.
\end{equation*}
\end{thm}

\begin{proof}
Observe that if $f(x)$ is log-concave, then also $f_{_tX}(x)=\frac{f(x)}{F(t)}$ is log-concave. It follows from Theorem 2.3 of Fradelizi et al. \cite{fradelizi} that if $X$ has a log-concave pdf, then $V(X)\leq 1$ and so the proof follows from this remark.
\end{proof}

\section{Proportional reversed hazard rate model}

In this section we consider the reversed proportional hazard rate model, see Gupta and Gupta \cite{gupta} for details. 

The family of random variables $\{X^{(a)}: a>0\}$ follows a proportional reversed hazard rate model if there exists a non-negative random variable $X$  with cdf $F$ and pdf $f$ such that
\begin{equation}
\label{eq22}
F^{(a)}(t)=\mathbb P(X^{(a)}\leq t)=[F(t)]^a, \ f^{(a)}(t)=a [F(t)]^{(a-1)} f(t), \ \ \ t>0.
\end{equation}
We remark that the model takes the name from the fact that the reversed hazard rate functions of the random variables in the family are proportional to the reversed hazard rate function of $X$:
\begin{equation*}
q^{(a)}(t)=\frac{f^{(a)}(t)}{F^{(a)}(t)}=a\ \frac{f(t)}{F(t)}=a\ q(t).
\end{equation*}
Moreover,  the cumulative reversed hazard rate function can be expressed as
\begin{equation*}
\Lambda^{(a)^*}(t)=-\log F^{(a)}(t)=a\Lambda^*(t).
\end{equation*}
The proportional reversed hazard rate model finds applications with parallel systems. In fact, if we have a system composed by $n$ units in parallel and characterized by i.i.d. lifetimes $X_1,\dots,X_n$ with cdf $F(t)$, the lifetime of the system is given by $X^{(n)}=\max\{X_1,\dots,X_n\}$. Then, we have $F_{X^{(n)}}(t)=[F(t)]^n$, i.e., the system satisfies the proportional reversed hazard rate model \eqref{eq22} with $a=n$.

In order to evaluate the past varentropy of $X^{(a)}$, we determine the past entropy of $X^{(a)}$. This can be expressed as
\begin{eqnarray*}
H\left(_t X^{(a)}\right)&=&-\Lambda^{(a)^*}(t)-\frac{1}{[F(t)]^a} \int_0^t f^{(a)}(x) \log f^{(a)}(x) \mathrm dx \\
&=& -a\Lambda^*(t)-\frac{1}{[F(t)]^a}\int_0^{[F(t)]^a}\gamma(y;a) \mathrm dy,
\end{eqnarray*}
with the change of variable $y=[F(x)]^a$, and where $\gamma(y;a)=\log\left[a y^{1-1/a}f(F^{-1}(y^{1/a}))\right]$. Hence, we obtain the past varentropy of $X^{(a)}$ as
\begin{eqnarray*}
V\left(_t X^{(a)}\right)&=&\frac{1}{[F(t)]^a}\int_0^t f^{(a)}(x) (\log f^{(a)}(x))^2 \mathrm dx -\left[\frac{1}{[F(t)]^a} \int_0^t f^{(a)}(x) \log f^{(a)}(x) \mathrm dx\right]^2 \\
&=&\frac{1}{[F(t)]^a}\int_0^{[F(t)]^a} [\gamma(y;a)]^2 \mathrm dy -\left[\frac{1}{[F(t)]^a} \int_0^{[F(t)]^a} \gamma(y;a) \mathrm dy\right]^2 .
\end{eqnarray*}

\section*{Conclusions}
In this paper, we have introduced and studied the past varentropy. It is related to the past entropy, which is a measure of information about the past lifetime distribution. In particular, the past varentropy provides the variability of the information given by the past entropy. We have showed several properties of the past varentropy, such as its behaviour under linear or monotonic transformations and what happens if it is constant. In addition, we have given bounds for the past varentropy and an application to the proportional reversed hazard rate model.

\acks
Francesco Buono and Maria Longobardi are partially supported by the GNAMPA research group of INdAM (Istituto Nazionale di Alta Matematica) and MIUR-PRIN 2017, Project "Stochastic Models for Complex Systems" (No. 2017 JFFHSH).

\end{document}